\documentclass[a4paper, 12pt, twoside, notitlepage,reqno]{amsart}

\usepackage{amsmath,amscd}
\usepackage{amssymb}
\usepackage{amsthm}
\usepackage{comment}
\usepackage{graphicx, xcolor}
\usepackage{mathrsfs}
\usepackage[ocgcolorlinks, linkcolor=blue]{hyperref}

\usepackage{bm}
\usepackage{bbm}
\usepackage{url}

\usepackage[utf8]{inputenc}
\usepackage{mathtools,amssymb}
\usepackage{esint}
\usepackage{tikz}
\usepackage{dsfont}
\usepackage{relsize}
\usepackage{url}
\usepackage{xcolor}
\usepackage{graphicx}
\usepackage{mathrsfs}
\usepackage[shortlabels]{enumitem}
\usepackage{lineno}
\usepackage{amsmath}
\usepackage{enumitem}
\usepackage{amsthm} 
\usepackage{fullpage} % wide page
\usepackage{verbatim}
\usepackage{dsfont}
\numberwithin{equation}{section}

\allowdisplaybreaks

\mathtoolsset{showonlyrefs}

\graphicspath{{images/}}

\newtheorem{theorem}{Theorem}[section]
\newtheorem{lemma}[theorem]{Lemma}
\newtheorem{definition}[theorem]{Definition}
\newtheorem{corollary}[theorem]{Corollary}
\newtheorem{conjecture}[theorem]{Conjecture}

\newtheorem{proposition}[theorem]{Proposition}
\newtheorem{question}{Question}
\newtheorem{remark}[theorem]{Remark}

\title[Nonabelian ray transform for twisted geodesics]{Loop group factorization method for the magnetic and thermostatic nonabelian ray transforms}
\author[S.R. Jathar]{Shubham R. Jathar}
\address{Indian Institute of Science Education and Research (IISER) Bhopal, India \& Computational Engineering, School of Engineering Sciences,
Lappeenranta-Lahti University of Technology LUT, Lappeenranta, Finland}
\email {shubham.jathar@lut.fi}

\author[M. Kar]{Manas Kar}
\address{Indian Institute of Science Education and Research (IISER) Bhopal, India}
\email{manas@iiserb.ac.in}

\author[J. Railo]{Jesse Railo}
\address{Department of Pure Mathematics and Mathematical Statistics,
University of Cambridge,
Cambridge, UK \& Computational Engineering, School of Engineering Sciences,
Lappeenranta-Lahti University of Technology LUT, Lappeenranta, Finland}
\email{jesse.railo@lut.fi}

\newcommand{\C}{{\mathbb C}}

\newcommand{\Z}{{\mathbb Z}}

\newcommand{\A}{{\mathcal A}}
\newcommand{\B}{{\mathcal B}}

\newcommand{\psm}{\partial_+SM}
\newcommand{\pnsm}{\partial_-SM}

\newcommand{\id}{\mathrm{Id}}

\newcommand{\fu}{\mathfrak u}

\NewDocumentCommand{\sff}{}{\mathrm{I\!I}}
\newcommand{\norm}[1]{\lVert #1 \rVert}
\newcommand{\ip}[2]{\left\langle #1,#2 \right\rangle}%inner product or duality pairing
\DeclareMathOperator{\inte}{int} %divergence

\begin{document}

	\begin{abstract} 
We study the injectivity of the matrix attenuated and nonabelian ray transforms on compact surfaces with boundary for nontrapping $\lambda$-geodesic flows and the general linear group of invertible complex matrices. We generalize the loop group factorization argument of Paternain and Salo to reduce to the setting of the unitary group when $\lambda$ has the vertical Fourier degree at most $2$. This covers the magnetic and thermostatic flows as special cases. Our article settles the general injectivity question of the nonabelian ray transform for simple magnetic flows in combination with an earlier result by Ainsworth. We stress that the injectivity question in the unitary case for simple Gaussian thermostats remains open. Furthermore, we observe that the loop group argument does not apply when $\lambda$ has higher Fourier modes.
		\medskip
		
		
  %44A12  	Radon transform
  %58J90  	Applications of PDEs on manifolds
  %37E35  	Flows on surfaces

	\end{abstract}
    \maketitle
% \tableofcontents
 \section{Introduction}

Our main goal in this article is to study the nonabelian $X$-ray transform for general connection $A$ and a matrix-valued field $\Phi$ for the $\lambda$-geodesic flows. We will work on Riemannian manifolds $(M, g)$ and, for all of our main results, we restrict our attention to Riemannian manifolds with $\operatorname{dim}(M)=2$. We define $G$ to be the Lie group of matrices with a Lie algebra $\mathfrak{g}$ and denote by $\Omega^1(M, \mathfrak{g})$ the space of smooth $\mathfrak{g}$-valued
$1$-forms. We will assume that $A\in \Omega^1(M, \mathfrak{g})$ and $\Phi \in C^{\infty}(M, \mathfrak{g})$. Formally, the quantities $A$ and $\Phi$ are called a Yang-Mills potential and Higgs field, respectively.
For a given pair $(A, \Phi)$, suppose that the map $U : [0,\tau(x,v)] \rightarrow G$ satisfies the matrix ordinary differential equation (ODE)
\begin{equation}\label{matrix_ODE}
\dot{U}+\left[A_{\gamma_{x, v}(t)}\left(\dot{\gamma}_{x, v}(t)\right)+\Phi\left(\gamma_{x, v}(t)\right)\right] U=0, \quad U(\tau(x, v))=\mathrm{Id}
\end{equation}
along the $\lambda$-geodesic $\gamma_{x, v}$ starting at the influx boundary point $(x,v)\in \partial_{+}SM$.
We define the scattering data (or the nonabelian $\lambda$-ray transform) along $\lambda$-geodesic $\gamma_{x, v}$ by requiring $C_{A,\Phi}^{\lambda}(\gamma_{x, v}) := U(0).$ The identity \eqref{matrix_ODE} reduces to 
\[
\log U(0)=\int_0^{\tau(x, v)}\left[A_{\gamma_{x, v}(t)}\left(\dot{\gamma}_{x, v}(t)\right)+\Phi\left(\gamma_{x, v}(t)\right)\right] d t
\]
if $A$ and $\Phi$ are assumed to be scalar $\mathbb{C}$-valued functions and this transformation can be regarded as the classical $X$-ray transform of $A+\Phi$ along the twisted geodesic flow. Note that $C_{A,\Phi}^{\lambda}$ maps $C_{A,\Phi}^{\lambda} : \partial_{+}SM \rightarrow G$.
\begin{question}
Can one recover the pair $(A,\Phi)$ from the knowledge of the scattering data $C_{A,\Phi}^{\lambda}$ corresponding to the $\lambda$-geodesic flow?
\end{question}
The nonabelian ray transform has appeared in many different areas of physics, geometry and applied mathematics. In the case when the Higgs fields $\Phi=0$, the scattering data $C_{A}$ represents the parallel transport of the connection $A$ and the injectivity for the nonabelian $X$-ray transform is related to the recovery of a connection up to a gauge from $C_A$ along the geodesics of the metric $g$. For the unitary connections $A$, the scattering data $C_A$ can be recovered from the connection Laplacian $d_{A}^{*}d_A$, where $d_A=d+A$, at a high energy limit, and also this data can be generated from the corresponding wave equation, see for instance \cite{Jakobson_Strohmaier:2007, oksanen2020inverse, Gunther:2004} for further details. When the connection $A=0$ and $\Phi\in C^{\infty}(M,\mathfrak{so}(3))$, the nonabelian  ray transform appears in polarimetric neutron tomography, see \cite{Desai_et_all:2020, hilger2018tensorial}, and in the case of general attenuation pairs $(A,\Phi)$, it has applications in solitons, see \cite{Manakov_Zakharov:1981, Ward:1988}. Other applications include coherent quantum tomography \cite{Ilmavirta:2016}. See also \cite{Novikov:2019} for further applications.  

There is a natural obstruction to solve the above question. Let us define the gauge group $\mathcal{G}$ consisting of smooth $u : M\rightarrow G$ such that $u|_{\partial M} = \text{Id}$ and we note that the gauge group $\mathcal{G}$ acts on pairs $(A,\Phi)$ by requiring 
$
(A, \Phi) \cdot u=\left(u^{-1} d u+u^{-1} A u, u^{-1} \Phi u\right)
$.
Observe that for any $u \in \mathcal{G}$, one has
$
C_{(A, \Phi) \cdot u}^{\lambda}=C_{A, \Phi}^{\lambda} .
$
 Therefore, the main interest of this article is to understand the injectivity question 
for the nonabelian twisted X-ray transform
$
(A, \Phi) \mapsto C_{A, \Phi}^{\lambda}
$
up to the natural gauge.
Throughout the article, we will denote the set of $n\times n$ matrices by $\mathbb{C}^{n\times n}$ or $\mathfrak{g l}(n, \mathbb{C})$. We can now state the following conjecture related to the uniqueness of the nonabelian ray transform up to the natural gauge for general connection and Higgs fields for $\lambda$-geodesic flows.
\begin{conjecture}\label{conjecture 1}
Let $(M, g, \lambda)$, $\lambda \in C^\infty(SM)$, be a compact Riemannian surface with strictly $\lambda$-convex boundary and nontrapping $\lambda$-geodesic flow. Let there be two pairs $(A, \Phi)$ and $(B, \Psi)$ such that $A, B \in \Omega^1(M, \mathfrak{g l}(n, \mathbb{C}))$ and $\Phi, \Psi \in C^{\infty}(M, \mathfrak{g l}(n, \mathbb{C}))$. If
$$
C_{A, \Phi}^{\lambda}=C_{B, \Psi}^{\lambda},
$$
then there exists $u \in \mathcal{G}$ such that $(A, \Phi) \cdot u=(B, \Psi)$.
\end{conjecture}
Conjecture \ref{conjecture 1} has been settled in \cite{Paternain:Salo:2020:a} when $\lambda = 0$ and with no conjugate point assumption. Notice that the gauge is trivial with the a priori assumption $A,B= 0$. The principle idea behind solving Conjecture \ref{conjecture 1} in the case of $\lambda=0$ is to reduce the problem into the attenuated $X$-ray transform for general connection and Higgs fields via a pseudolinearization  identity and then study this attenuated $X$-ray transform. In particular, they proved Conjecture \ref{conjecture 2} for $\lambda=0$. To proceed further, we consider the unique solution $u(t)$ for the vector-valued ordinary differential equation
$$
\dot{u}+\left[A_{\gamma_{x, v}(t)}\left(\dot{\gamma}_{x, v}(t)\right)+\Phi\left(\gamma_{x, v}(t)\right)\right] u=-f\left(\gamma_{x, v}(t), \dot{\gamma}_{x, v}(t)\right), \quad u(\tau(x, v))=0,
$$
where ${\gamma}_{x, v}(t)$ is a $\lambda$-geodesic with $\gamma_{x,v}(0)=x$, $\dot\gamma_{x,v}(0)=v$ and $A \in \Omega^1(M, \mathfrak{g l}(n, \mathbb{C})), \Phi \in C^{\infty}(M, \mathfrak{g l}(n, \mathbb{C}))$,
$f \in C^{\infty}\left(S M, \mathbb{C}^n\right)$ are given known quantities. 
The attenuated twisted $X$-ray transform of $f$ can be defined by
$$
I_{A, \Phi}^{\lambda}(f)(x, v):=u(0)
$$
where $(x, v) \in \partial_{+} S M$. We now formulate another conjecture.
\begin{conjecture}\label{conjecture 2}
Let $(M, g, \lambda)$, $\lambda \in C^\infty(SM)$, be a compact Riemannian surface with strictly $\lambda$-convex boundary and nontrapping $\lambda$-geodesic flow. Let $(A, \Phi)$ be an arbitrary attenuation pair with $A \in \Omega^1(M, \mathfrak{gl}(n, \mathbb{C}))$ and $\Phi \in C^{\infty}(M, \mathfrak{gl}(n, \mathbb{C}))$. Let $f=f_0+\alpha \in C^\infty(SM,\C^n)$ be a smooth function where $f_0\in C^{\infty}(M,\C^n)$ and $\alpha$ is a $\C^n$-valued $1$-form. If $I_{A, \Phi}^{\lambda}(f)=0$, then $f_0=\Phi p$ and $\alpha=dp+A p$ for some $p \in C^\infty(M,\C^n)$ with $\left.p\right|_{\partial M}=0.$
\end{conjecture}
In \cite{Paternain:Salo:2020:a}, Conjecture \ref{conjecture 2} has been proved in the case of $\lambda=0$ and absence of conjugate points. They used a factorization theorem for Loop groups to reduce the problem of attenuated $X$-ray transform for general connections and Higgs fields to the problem of attenuated $X$-ray transform for unitary connections and skew-Hermitian Higgs fields, which is precisely Conjecture \ref{conjecture 3} for $\lambda=0$. Conjectures \ref{conjecture 1} and \ref{conjecture 3} were proved in \cite{Paternain:Salo:Uhlmann:2012} for unitary connections $A$ and skew-Hermitian Higgs fields $\Phi$ for usual geodesic flows. In \cite{paternain2021carleman}, Conjecture \ref{conjecture 2} has been proved for $\lambda=0$ in the case of compact Riemannian nontrapping negatively curved manifold with smooth strictly convex boundary for general connections and Higgs fields, and they used certain Carleman estimates related to the geodesic flow. This result has been extended to the case of attenuated ray transform for Gaussian thermostat on a compact Riemannian surface with negative thermostat curvature in \cite{Assylbekov:Rea:Arxiv:2021}. In particular, they proved an injectivity result for thermostat ray transform with a general connection and Higgs fields up to a gauge transformation under curvature assumptions.
We denote by $\mathfrak{u}(n)$ the Lie algebra of the unitary group. We now state the following conjecture. 
\begin{conjecture}\label{conjecture 3}
Let $(M, g, \lambda)$, $\lambda \in C^\infty(SM)$, be a compact Riemannian surface with strictly $\lambda$-convex boundary and nontrapping $\lambda$-geodesic flow. Let $A: S M \rightarrow \mathfrak{u}(n)$ be a unitary connection and $\Phi: M \rightarrow \mathfrak{u}(n)$ be a skew-Hermitian matrix function. Let $f=f_0+\alpha \in C^\infty(SM,\C^n)$ be a smooth function where $f_0\in C^{\infty}(M,\C^n)$ and $\alpha$ is a $\C^n$-valued $1$-form. If $I_{A, \Phi}^{\lambda}(f)=0$, then $f_0=\Phi p$ and $\alpha=d p+A p$ for some $p \in C^\infty(M,\C^n)$ with $\left.p\right|_{\partial M}=0.$
\end{conjecture}
Conjecture \ref{conjecture 3} has been established in the case of $\lambda=0$ and no conjugate point assumption, see  \cite[Theorem 1.3]{Paternain:Salo:Uhlmann:2012}. In addition to that, they also proved Conjecture \ref{conjecture 1} for Hermitian connections and skew-Hermitian Higgs fields in a compact simple Riemannian surface in the case of usual geodesics, see for instance \cite[Theorem 1.5]{Paternain:Salo:Uhlmann:2012}. The idea of the proof of Conjecture \ref{conjecture 3} for $\lambda=0$ in \cite{Paternain:Salo:Uhlmann:2012} is based on a proof similar to Kodaira vanishing theorem in complex geometry to the transport problem relevant for the corresponding nonabelian X-ray transform. In particular, they used a scalar holomorphic integrating factor and applied a suitable Pestov identity. A similar approach has been followed in \cite[Theorem 1.2]{Ainsworth:2013} to prove Conjecture \ref{conjecture 3} in the presence of magnetic geodesic flows.   
So far, \cite[Theorem 1.2]{Ainsworth:2013} is the best-known result when $\lambda$ is assumed to be a function on $M$.

\subsection{Main results}
In this article, we first show that Conjecture \ref{conjecture 2} implies Conjecture \ref{conjecture 1} via a pseudolinearization identity for nontrapping $\lambda$-geodesic flows. The main contribution of this article is to find out under what conditions on $\lambda$, Conjecture \ref{conjecture 3} implies Conjecture \ref{conjecture 2}. We also state here interesting conclusions on the generalization of the results in \cite{Ainsworth:2013} for the simple magnetic flows on surfaces with boundary.

We now state our main theorems in this article and some of its corollaries related to the nonabelian ray transform. To do that, let us first introduce 
$\Omega_{k}=C^{\infty}(SM,\C^n)\cap H_{k}$, where $H_k$ is the eigenspace of $-iV$ corresponding to the eigenvalue $k$ and $V$ denotes the vertical vector field. In more detail, we will discuss such function spaces and related vector fields in Section \ref{sec_preliminary}. Our first theorem is an application of pseudolinearization  identity to prove the uniqueness of the nonabelian ray transform up to the natural gauge for general $\lambda$ whenever one assumes the injectivity result for the attenuated ray transform for general connections and Higgs fields for $\lambda\in C^{\infty}(SM)$ up to a natural obstruction.
\begin{theorem}\label{conj2_conj1}
Let $(M, g, \lambda)$, $\lambda \in C^\infty(SM)$, be a compact Riemannian surface with strictly $\lambda$-convex boundary and nontrapping $\lambda$-geodesic flow. If Conjecture \ref{conjecture 2} holds for $\lambda$, then Conjecture \ref{conjecture 1} is true for $\lambda$. 
\end{theorem}
The next theorem shows that Conjecture \ref{conjecture 3} implies Conjecture \ref{conjecture 2} whenever the vertical Fourier degree of $\lambda$ is at most $2$.
\begin{theorem}\label{them_conjecture 2}
Let $(M, g, \lambda)$, $\lambda \in C^\infty(SM)$, be a compact Riemannian surface with strictly $\lambda$-convex boundary and nontrapping $\lambda$-geodesic flow. If $\lambda \in \bigoplus_{-2 \leq k \leq 2} \Omega_k$ and Conjecture \ref{conjecture 3} holds for $\lambda$, then Conjecture \ref{conjecture 2} is true for $\lambda$.      
\end{theorem}
Of course, trivially for any $\lambda \in C^\infty(SM)$ Conjecture \ref{conjecture 2} implies Conjecture \ref{conjecture 3}. In the proof of Theorem \ref{them_conjecture 2}, we mainly employ the loop group factorization arguments to reduce the attenuated ray transform for general connections and Higgs fields to the attenuated ray transform for the unitary connections and skew-Hermitian Higgs fields when $\lambda$ is assumed to be the sum of a function and $1$-form. 
However, in the case where $\lambda\in \oplus_{k=-j}^{k=j}\Omega_{k}$, for $j\geq 3$, the argument based on the loop group factorization no longer directly works (cf. Remark \ref{remark:obstruction}).  
\begin{corollary}\label{them_conjecture_therm 1} 
Let $(M, g, \lambda)$, $\lambda \in C^\infty(SM)$, be a compact Riemannian surface with strictly $\lambda$-convex boundary and nontrapping $\lambda$-geodesic flow. Let $\lambda \in \bigoplus_{-2 \leq k \leq 2} \Omega_k$ and assume that Conjecture \ref{conjecture 3} holds for $\lambda$. Suppose we are given pairs $(A, \Phi)$ and $(B, \Psi)$ with $A, B \in \Omega^1(M, \mathfrak{g l}(n, \mathbb{C}))$ and $\Phi, \Psi \in C^{\infty}(M, \mathfrak{g l}(n, \mathbb{C}))$. If
$$
C_{A, \Phi}^{\lambda}=C_{B, \Psi}^{\lambda},
$$
then there is $u \in \mathcal{G}$ such that $(A, \Phi) \cdot u=(B, \Psi)$.
\end{corollary}
\noindent
The above corollary is a direct consequence of Theorem \ref{conj2_conj1} and Theorem \ref{them_conjecture 2}.

In the case of magnetic geodesics, we can prove the following injectivity result for the attenuated ray transform for general connections and Higgs fields utilizing \cite[Theorem 1.2]{Ainsworth:2013} and loop group factorization arguments summarized in the earlier discussions and results. We say that $(M,g,\lambda)$ is a simple magnetic surface if $\lambda \in C^{\infty}(M)$, $M$ is compact with smooth strictly magnetic convex boundary and the magnetic flow is nontrapping without conjugate points.
\begin{theorem}[Uniqueness for the attenuated magnetic ray transform]\label{them_conjecture 3}
Let $(M,g,\lambda)$ be a simple magnetic surface. Let $(A,\Phi)$ be an arbitrary attenuation pair $(A, \Phi)$ with $A \in \Omega^1(M, \mathfrak{g l}(n, \mathbb{C}))$ and $\Phi \in C^{\infty}(M, \mathfrak{g l}(n, \mathbb{C}))$. Let $f=f_0+\alpha \in C^\infty(SM,\C^n)$ be a smooth function where $f_0\in C^{\infty}(M,\C^n)$ and $\alpha$ is a $\C^n$-valued $1$-form. If $I_{A,\Phi}^{\lambda}f=0$, then $f_0=\Phi p$ and $\alpha=dp+Ap$ for some $p \in C^\infty(M,\C^n)$ with $\left.p\right|_{\partial M}=0.$
\end{theorem}
\begin{corollary}[Uniqueness for the nonabelian magnetic ray transform]\label{them_conjecture 1}
Let $(M,g,\lambda)$ be a simple magnetic surface. Suppose we are given pairs $(A, \Phi)$ and $(B, \Psi)$ with $A, B \in \Omega^1(M, \mathfrak{g l}(n, \mathbb{C}))$ and $\Phi, \Psi \in C^{\infty}(M, \mathfrak{g l}(n, \mathbb{C}))$. If
$$
C_{A, \Phi}^{\lambda}=C_{B, \Psi}^{\lambda},
$$
then there is $u \in \mathcal{G}$ such that $(A, \Phi) \cdot u=(B, \Psi)$.
\end{corollary}
\noindent 
The above corollary is a direct consequence of Corollary \ref{them_conjecture_therm 1} and Theorem \ref{them_conjecture 3}.

\subsection{Further discussion on the earlier literature}
There is an extensive literature to study the nonabelian ray transform in the case of $\lambda=0$. For the usual geodesic case, Conjectures \ref{conjecture 1} and \ref{conjecture 2} are proved in $\mathbb{R}^3$, see \cite{Novikov:2002}, and in \cite{Paternain_Salo_Uhl_Zhou:2019} in the case of compact manifolds with strictly convex boundary of dimensions three and higher, and admitting a strictly convex function. This is based on an approach for inverting the local geodesic ray transform \cite{Uhlmann_Vasy:2016}. For the pair $(A,\Phi)$ taking value in $\mathfrak{u}(n)$, Conjectures \ref{conjecture 1} and \ref{conjecture 2} are proved in \cite{Paternain:Salo:Uhlmann:2012} when $\lambda=0$ on a simple nontrapping Riemannian surface. In \cite{Novikov:2002}, Novikov proved the local uniqueness results in $\mathbb{R}^n, n\geq 2$ under the assumption that the attenuation pairs $(A,\Phi)$ are not necessarily compactly supported with certain decay conditions at infinity. However, he constructed an example to show that the global uniqueness result fails in general. Later, a similar to Conjecture \ref{conjecture 1}, an injectivity question for the nonabelian Radon transform has been considered by Eskin in \cite{Eskin:2004} for compactly supported attenuated pairs $(A,\Phi)$ and the main idea of their proof uses the existence of matrix holomorphic integrating factors in the vertical variable. 
We also mention the work of Finch and Uhlmann \cite{Finch_Uhlmann:2001} that proves Conjecture \ref{conjecture 1} in the domains of $\mathbb{R}^2$ for the connections having small curvature.

Under certain restrictions on the size of the attenuation pairs $(A, \Phi)$, the nonabelian ray transform has been first considered in \cite{Vertge:1991, Vertge:2000, Sharafut:2000}
in a manifold setting.
In \cite{Sharafut:2000}, Sharafutdinov established Conjecture \ref{conjecture 1} for the case $\Phi=0$ and assuming that the connections are $C^1$ close to another connection with small curvature. Conjecture \ref{conjecture 2} was proved in \cite{Salo_Uhlmann:2011} for $A=0$ and $n=1$ when $\lambda =0.$
Generic injectivity and stability issues are discussed and some Fredholm alternatives are considered for connections in \cite{Monard_Paternain:2020, Zhou:2017}.
Conjecture \ref{conjecture 1} is proved under negative curvature assumption for $\lambda=0$ in \cite{paternain2021carleman} for $n \geq 2$.
In \cite{Bohr_Paternain:2023}, the authors established the existence of matrix holomorphic integrating factors for any simple surface where they provided a full characterization of the range of the nonabelian X-ray transform.
For the closed Riemannian surfaces, the nonabelian ray transform problems are considered in \cite{Paternain:2013}.

Compared to the usual geodesic case, not much is known in the nonabelian  ray transform literature for $\lambda \neq 0$. In \cite{Ainsworth:2013}, Ainsworth proved Conjectures \ref{conjecture 1} and \ref{conjecture 3} for the magnetic geodesics on simple surfaces with a unitary connection and a skew-Hermitian matrix Higgs field. Recently, Conjectures \ref{conjecture 1} and \ref{conjecture 2} have been established on a compact Riemannian surface under the assumption of negative thermostat curvature in \cite{Assylbekov:Rea:Arxiv:2021}. We are not aware of any studies on the nonabelian and attenuated ray transforms in higher dimensions for more general families of curves, and this appears to be an interesting direction for future research. Finally, we mention that many other integral geometry problems are studied extensively for magnetic and thermostatic flows \cite{Dairbekov:Paternain:2007:Entropy1,Dairbekov:Paternain:2007:II,Dairbekov-Paternain-SU-2007-magnetic-rigidity,Ainsworth:2013,Ainsworth-Assylbekov-2015-magnetic-range,Ainsworth:2015,Assylbekov:Zhou:2017,Assylbekov:Dairbekov:2018,Assylbekov:Rea:Arxiv:2021,mazzucchelli2023general,Muñoz-Thon_2024,cuesta2024smooth,MR4762978}. For a recent survey of inverse problems related to twisted geodesic flows, we refer to \cite{jathar2024inverse}.

\subsection{Organization of the article} In Section \ref{sec_preliminary}, we introduce some basic notations, definitions and properties of twisted geodesic flows on surfaces that will be required in the later sections. Section \ref{sec-ATT} deals with the attenuated and nonabelian ray transform for twisted geodesics on surfaces. The proofs of main results are given in Section \ref{se_proof}. In Appendices \ref{sec:appendix-nontrappingness} and \ref{sec_app2}, we consider some auxiliary results for nontrapping $\lambda$-geodesic flows and the smoothness of the primitive for the kernel of the (matrix) attenuated $\lambda$-ray transforms, respectively.

\section{Preliminaries}\label{sec_preliminary}
Throughout this article, we mainly follow \cite{GIP2D} for the used notations, definitions and some relevant properties. Similar concepts appear earlier in many works, for example, in \cite{Guillemin:Kazhdan:1980, Singer_Thorpe:1976}. Basics related to the twisted geodesic flows can be recalled from \cite{Marry:Paternain:2011:notes}.

Assume $(M,g)$ to be a compact Riemannian surface with smooth boundary and $g$ being a Riemannian metric. Let $SM=\{(x,v)\in TM: \langle v,v\rangle_g=1\}$ be the unit tangent bundle. We denote $\partial_{+}SM$ by the influx boundary where the tangent vectors point inside the manifold and $\partial_{-}SM$ by the outflux boundary where the tangent vectors point outside the manifold. More precisely, these sets are defined by
\begin{equation}\label{out_in_flux}
\partial_{\pm}SM := \{(x,v)\in SM : x\in \partial M, \ \pm \langle v,\nu(x)\rangle_g\geq 0\}  
\end{equation}
where $\nu$ is the inward unit normal to $\partial M.$
For any $\lambda\in C^{\infty}(SM)$, a curve $\gamma$ is called a  $\lambda$-geodesic (or twisted geodesic) if it satisfies the $\lambda$-geodesic equation 
\begin{equation}\label{lam_twis_eq}
  D_{t}\dot{\gamma} = \lambda(\gamma,\dot{\gamma})i\dot{\gamma}  
\end{equation}
where $i$ is the counterclockwise rotation by $90$ degrees \cite[Definition 7.1]{Marry:Paternain:2011:notes}. In general, the term $\lambda(\gamma,\dot{\gamma})i\dot{\gamma}$ represents an external force that pushes a particle out from the usual geodesic trajectory, corresponding to the case $\lambda=0$.
We say that a curve $\gamma$ is a magnetic geodesic if it satisfies
\[
D_t\dot{\gamma}=Y(\dot{\gamma}),
\]
where the Lorentz force $Y: TM \to TM$. The magnetic field $\Omega$, being a closed $2$-form on $M$, is related to the Lorentz force via the unique bundle map
\begin{equation}
\Omega_x(\xi,\eta)=\ip{Y_x(\xi)}{\eta}_g, \quad \forall x \in M, \ \xi,\eta \in T_xM.
\end{equation}
The curve is said to be Gaussian thermostatic (or thermostatic geodesic) if one has the identity
\[
D_t \dot{\gamma}=E(\gamma)-\frac{\langle E(\gamma), \dot{\gamma}\rangle}{|\dot{\gamma}|^2} \dot{\gamma},
\]
where $E$ is a smooth vector field on $M$.
In the case of magnetic geodesics, $\lambda(x,v)$ only depends on $x$. In the case of Gaussian thermostats, $\lambda(x,v)$ depends linearly on $v$.

By abuse of notation, we continue to write $\gamma_{x,v}$ for the twisted geodesic starting at $x$ with the velocity $v$ where $(x,v)\in SM$. We define the $\lambda$-geodesic flow $\phi_t : SM \rightarrow SM$ by
\[
\phi_t(x,v) := (\gamma_{x,v}(t), \dot{\gamma}_{x,v}(t)),
\]
where $\gamma_{x,v}$ is a $\lambda$-geodesic. In our main results, we assume that the manifold $(M,g)$ has nontrapping $\lambda$-geodesic flow, which means that the forward and backward exit times $$\tau^{\pm}(x,v) := \inf\{\, t \geq 0 \,:\, \phi_{\pm t}(x,v) \in \partial SM\,\}$$ of the $\lambda$-geodesics are finite for all $(x,v)\in \inte SM$ and $\tau^{\pm}$ can be extended continuously to $\partial SM$. We also make a strict $\lambda$-convexity assumption on the boundary of the manifold. The boundary $\partial M$ is strictly $\lambda$-convex at a point $x \in \partial M$ if the following inequality holds: 
$$
\sff_x(v, v) >-\left\langle \lambda(x,v)iv, \nu(x)\right\rangle \quad \text { for all } v \in S_x(\partial M),
$$
where $\sff$ is the second fundamental form of the boundary $\partial M.$ We refer the reader to \cite{Jathar:Kar:Railo:2023:arxiv} for a further discussion on strict $\lambda$-convexity. If $(M,g,\lambda)$ is strictly $\lambda$-convex and nontrapping without conjugate points, then we say that the $\lambda$-geodesic flow is simple.

By \cite[Lemma 7.4]{Marry:Paternain:2011:notes}, the infinitesimal generator of $\phi_t$ can be written as $ X +\lambda V,$ where $X$ denote the usual geodesic vector field acting on the unit circle bundle $SM$ and $V$ is the vertical vector field.
Denote $X_{\perp}:=[X,V]$ and that allows one to define $\{X,X_{\perp}, V\}$ as a global frame of $T(SM)$. The following structure equations hold: \begin{equation}\label{eq: basic commutators}
    [X,V] = X_\bot, \quad [X_\bot,V] =-X,\quad [X,X_\bot] =-KV
\end{equation}
where $K$ is the Gaussian curvature of the Riemannian surface $(M,g)$, see \cite[Lemma 3.5.5]{GIP2D}. Since $\{X,-X_\bot,V\}$ forms a positively oriented orthonormal frame, it follows that the unit sphere bundle $SM$ possesses a unique Riemmanian metric $G$, called the Sasaki metric. 
\begin{remark} We note that in two dimensions, if you have a smooth flow $\phi_t(x,v) =(\gamma_{x,v}(t),\dot{\gamma}_{x,v}(t))$ on $SM$, given by a smooth family of unit speed curves $(x,v) \mapsto \gamma_{x,v}$ on $M$, then the generator of $\phi_t$ is given as $X+\lambda V$ for some $\lambda \in C^\infty(SM)$ \cite[Theorem 1.4]{Assylbekov:Dairbekov:2018}. Thus, our setting in general covers all families of smooth unit speed curves on $M$. There are more general flows on $SM$, but they may not be natural for studying ray transforms on $M$.
\end{remark}

The volume form of the metric satisfies $dV_G = d\Sigma^3$, see \cite[Lemma 3.5.11]{GIP2D}. Here $d\Sigma^3$ represents a measure on $SM$, defined by 
\begin{equation}
    d\Sigma^3 := dV^2 \wedge dS_x
\end{equation}
is usually called the Liouville form, where $dV^2$ is the volume form of $(M,g)$ and $dS_x$ is the volume form of $(S_xM,g_x)$ for any $x \in M$.
For any $u,v:SM\to \C^n$, we define the following $L^2$-inner product
\[(u,v) =\int_{SM}\langle u,v\rangle_{\mathbb \C^n}\,d\Sigma^3.\]
We also define the spaces of eigenvectors of $V$ by requiring 
\begin{equation}\label{eq:GK-eig}
    H_k := \{\,u \in L^2(SM,\C^n)\,;\, -iVu = ku\,\}, \quad \Omega_k := \{\,u\in C^\infty(SM,\C^n)\,;\, -iVu=ku\,\}
\end{equation}
for any integer $k \in \Z$. Now we remark that the space $L^{2}(SM,\C^n)$ has a following orthogonal decomposition of direct sum 
\[L^{2}(SM,\C^n)=\bigoplus_{k\in\mathbb Z}H_{k}.\]
For any $u \in L^2(SM)$, there is a unique $L^2$-orthogonal decomposition
\begin{equation}\label{eq:orthogonal}
    u=\sum_{k=-\infty}^\infty u_k, \quad \norm{u}^2 = \sum_{k=-\infty}^\infty \norm{u_k}^2, \quad u_k \in H_k.
\end{equation}
 If $u \in C^\infty(SM)$, then each $u_k \in C^\infty(SM)$ and the series converges in the $C^\infty(SM)$ norm.

For any $I\subset\mathbb{Z}$, we denote by $\oplus_{k\in I}\Omega_{k}$ the set of smooth functions $u$ satisfying $u_{k}=0$ whenever $k\notin I$. Note that the vector fields $X$ and $V$ have the following mapping properties:
\[X:\oplus_{k\geq 0}\Omega_{k}\to \oplus_{k\geq -1}\Omega_{k},\qquad V:\Omega_{k}\to \Omega_{k}\]
see for instance \cite{GIP2D}. 

If we assume that $\lambda\in C^{\infty}(SM, \mathbb{C})$ has finite degree $m$. Then $\lambda$ can be written as
\[
\lambda=\sum_{k=-m}^{m} \lambda_k \ \in \Omega_{-m}\oplus \cdots \oplus \Omega_{m}.
\]
%Notice that, if $u\in\Omega_k$ and $\lambda_{\pm m}\in\Omega_{\pm m}$, then one could see that
%\[
%-iV(\lambda_{\pm m}Vu)
%= \lambda_{\pm m}(-iV(Vu)) + (-iV(\lambda_{\pm m}))Vu
%= (k\pm m)\lambda_{\pm m}Vu.
%\]
Notice that $\lambda_{\pm m} V$ has the following mapping properties 
\[
\lambda_{\pm m} V : \Omega_k \rightarrow \Omega_{k\pm m}, \quad k \neq 0,
\]
and $\lambda_{\pm m} V|_{\Omega_0}=0$.
We remark that above smooth functions may take values in complex matrices $\mathfrak{gl}(n,\C)$.

\begin{definition} A function $u:SM\to\C^n$ is said to be (fibrewise) holomorphic if $u_{k}=0$ for all $k<0$. Similarly, $u$ is said to be (fibrewise) antiholomorphic if $u_{k}=0$ for all $k>0$.
\end{definition}
\section{Attenuated and nonabelian twisted ray transforms}\label{sec-ATT}
 Let $(M, g, \lambda), \lambda \in C^{\infty}(S M)$, be a compact Riemannian surface with strictly $\lambda$-convex boundary and nontrapping $\lambda$-geodesic flow. Assume $\mathcal A\in C^{\infty}\left(S M, \mathbb{C}^{n \times n}\right)$. We introduce a closed manifold $(N,g)$ so that $(M,g)$ is isometrically embedded in $(N,g)$ and we extend $\mathcal{A}$ and $\lambda$ smoothly to whole of $N$. On the closed manifold $(N,g)$, the matrix valued function $\mathcal{A}$ defines a smooth cocycle $C$ over the $\lambda$-geodesic flow $\phi_{t}$. We call $C: S N \times \mathbb R \to G L(n, \mathbb{C})$ as a smooth cocycle if it satisfies the matrix ordinary differential equation 
\begin{equation}
    \frac{\mathrm{d} }{\mathrm{d} t}C(x, v, t)+\mathcal{A}\left(\phi_t(x, v)\right) C(x, v, t)=0, \quad C(x, v, 0)=\mathrm{Id}
\end{equation}
along the orbits of $\lambda$-geodesic flows. From the uniqueness of solution of ODEs and $\phi_t$ is $\lambda$-geodesic flow, the function $C$ satisfies 
\begin{align}\label{eq:cocycle:property}
C(x,v,t+s)&=C(\phi_t(x,v),s)C(x,v,t)
\end{align}
 for all $(x,v)\in SN$ and $t,s\in \mathbb R$. The identity \eqref{eq:cocycle:property} explains the name cocycle. 
 
Let $(M, g, \lambda)$, with $\lambda \in C^{\infty}(S M)$, be a compact Riemannian surface with strictly $\lambda$-convex boundary and nontrapping $\lambda$-geodesic flow. From Lemma \ref{lm:open:nontrapped}, the existence of nontrapping $\lambda$-geodesic flow is an open condition. Similarly, from the definition of strictly $\lambda$-convexity, it is also an open condition. Hence, there exists a slightly larger manifold $(M_0, g, \lambda_0)$ engulfing $M$, such that $\lambda_0$ is a smooth extension of $\lambda$ on $M_0$, and $M_0$ has strictly $\lambda_0$-convex boundary and the $\lambda_0$-geodesic flow is nontrapping.
 \begin{definition}\label{def:smooth:lambda:extension}
    Let $(M, g, \lambda)$, $\lambda \in C^{\infty}(SM)$, be a compact Riemannian surface with strictly $\lambda$-convex boundary and nontrapping $\lambda$-geodesic flow. Suppose there exists a larger compact manifold $(M_0, g_0, \lambda_0)$ containing $M$ such that:
    \begin{itemize}
        \item $\lambda_0 \in C^{\infty}(SM_0)$ is a smooth extension of $\lambda$ to $M_0$,
        \item the Riemannian metric $g_0$ is a smooth extension of $g$ to $M_0$
        \item $M_0$ has strictly $\lambda_0$-convex boundary,
        \item $M_0$ has nontrapping $\lambda_0$-geodesic flow.
    \end{itemize}
    Then $(M_0, g_0, \lambda_0)$ is called as a \emph{smooth strictly convex, nontrapping, extension} of $(M, g, \lambda)$.
\end{definition}
Next lemma is a generalisation of \cite[Lemma 5.3.2]{GIP2D} for the cocycle $C$ corresponding to the $\lambda$-geodesic flows $\phi_t$.
\begin{lemma}\label{lm:R:equation}
    Let $(M, g, \lambda)$, $\lambda \in C^{\infty}(S M)$, be a compact Riemannian surface with strictly $\lambda$-convex boundary and nontrapping $\lambda$-geodesic flow. Let $(M_0,g,\lambda_0)$ be a smooth strictly convex, nontrapping, extension as defined in Definition \ref{def:smooth:lambda:extension}. and $\tau_0$ the forward travel time in $M_0$. Define a function $R : SM \to GL(n, \mathbb{C})$ such that
\[R(x,v):=\left[C\left(x, v, \tau_{0}(x, v)\right)\right]^{-1}.\]
Then $R$ is a smooth function and satisfies
\begin{align}\label{eq:R:equation}
    (X+\lambda V)R+\mathcal{A}R=0,\qquad (X+\lambda V)R^{-1}-R^{-1}\mathcal{A}=0.
\end{align}
\end{lemma}
\begin{proof} Recall that $\tau_0|_{SM}$ is a smooth function on $S M$ (cf. \cite[Lemma 4.16]{Jathar:Kar:Railo:2023:arxiv}), and the cocycle is also a smooth function. Therefore, the function $R$ is smooth on $S M$.
Replacing $\tau_0(\phi_t(x,v))=\tau_0(x,v)-t$ in \eqref{eq:cocycle:property}, we can write
\begin{align*}
    C\left(x, v, t+\tau_0-t\right)=C\left(\varphi_t(x, v), \tau_0(x,v)-t\right) C(x, v, t),
\end{align*} 
for all $t\in [0,\tau_{0}(x,v)]$ and simplifying further yields 
\begin{align*}
    \left[ C\left(x, v, \tau_0(x,v)\right)\right]\left[C(x, v, t)\right]^{-1}=[C\left(\varphi_t(x, v), \tau_0({x,v})-t\right)].
\end{align*}
By the definition of $R$, we obtain
\begin{align*}
    R(\phi_t(x,v)) &= \left[C(\phi_t(x,v), \tau_0(\phi_t(x,v)))\right]^{-1} = \left[C(\phi_t(x,v), \tau_0(x,v) - t)\right]^{-1} \\
    &= C(x, v, t) \left[C\left(x, v, \tau_0(x,v)\right)\right]^{-1} = C(x, v, t) R(x,v).
\end{align*}
Differentiating the above identity with respect $t$ and evaluating at $t=0$, we have
\begin{align*}
 (X+\lambda V)R(x,v)&=\left.\frac{d}{dt} R(\phi_t(x,v)) \right|_{t=0}  =\left.\frac{d}{dt}C(x, v, t)\right|_{t=0}R(x,v)\\
 &=-\left.\mathcal{A}\left(\phi_t(x, v)\right) C(x, v, t)\right|_{t=0} R(x,v)\\
 &=-\mathcal{A}(x,v)R(x,v).
\end{align*}
To prove the other identity, we first compute
% \begin{align*}
%   0 &= \frac{d}{dt} \left(R(\phi_t(x,v)) \left(R(\phi_t(x,v)) \right)^{-1}   \right)\\
%   &=\left(\frac{d}{dt} R(\phi_t(x,v) )   \right)\left(R(\phi_t(x,v) )\right)^{-1}+R(\phi_t(x,v) ) \left(\frac{d}{dt} (R(\phi_t(x,v) )^{-1}   \right).
% \end{align*}
\begin{align*}
  0 &= \frac{d}{dt} \left( R(\phi_t(x,v)) \left( R(\phi_t(x,v))\right)^{-1} \right) \\
    &= \left( \frac{d}{dt} R(\phi_t(x,v)) \right) \left(R(\phi_t(x,v))\right)^{-1} + R(\phi_t(x,v)) \left( \frac{d}{dt} \left(R(\phi_t(x,v))\right)^{-1} \right)
\end{align*}
Then we have
% $$
% \begin{aligned}
% (X+\lambda V) R^{-1}(x, v) & =(\left.\frac{d}{d t}\left(R\left(\phi_t(x, v)\right)^{-1})\right|_{t=0}\right. \\
% & =-(R(x, v))^{-1}(X+\lambda V)R(x, v)(R(x, v))^{-1} \\
% & =(R(x, v))^{-1} \mathcal{A}(x, v),
% \end{aligned}
% $$
\begin{align*}
(X + \lambda V) R^{-1}(x, v) &= \left. \frac{d}{dt} \left( R\left( \phi_t(x, v) \right) \right)^{-1} \right|_{t=0} \\
&= - R^{-1}(x, v) (X + \lambda V) R(x, v) R^{-1}(x, v) \\
&= R^{-1}(x, v) \mathcal{A}(x, v),
\end{align*}
which establishes the formula.
\end{proof}
Note that the solution of equations \eqref{eq:R:equation} in the above lemma are usually called the matrix integrating factor.
\begin{definition}[Nonabelian $\lambda$-ray transform]\label{def:scat} Let $(M, g, \lambda)$, $\lambda \in C^{\infty}(S M)$, be a compact Riemannian surface with strictly $\lambda$-convex boundary and nontrapping $\lambda$-geodesic flow.
The $\lambda$-scattering data of $\mathcal A$ is the map 
\[C_{\A}^{\lambda}:\partial_+SM\to GL(n,\C)\]
defined by $C_{\A}^{\lambda}:=U_{+}|_{\partial_+SM}$ where $U_{+}(x,v) := [C(x,v,\tau(x,v))]^{-1}$ satisfies 
\begin{align*}
    (X+\lambda V)U_{+}+\A U_{+}=0,\qquad U_{+}|_{\partial_-SM}=\operatorname{Id}.
\end{align*}
We call $C_{\A}^{\lambda}$ the nonabelian $\lambda$-ray transform of $\A$.
\end{definition}
It is immediate to see that $C_{\mathcal{A}}^{\lambda} \in C^{\infty}\left(\partial_{+} S M, \mathbb{C}^{n \times n}\right)$.
\begin{definition}[Attenuated $\lambda$-ray transform]\label{def:aLray} Let $(M, g, \lambda)$, $\lambda \in C^{\infty}(S M)$, be a compact Riemannian surface with strictly $\lambda$-convex boundary and nontrapping $\lambda$-geodesic flow. The attenuated $\lambda$-ray transform of $f\in C^{\infty}(SM,\C^n)$ is defined by 
\[
I_{\A}^{\lambda}f:=u^f|_{\partial_+SM}
\]
where $u^f:SM\to \C^n$ is a solution to the following transport equation
   \begin{equation}\label{att_trans_eq}
       (X+\lambda V)u^f+\A u^f=-f \text{ in } SM,\qquad u^f|_{\partial_-SM}=0.
   \end{equation}
\end{definition}
Next lemma is a generalisation of \cite[Lemma 3.4]{Paternain:Salo:2020:a}
to the case of $\lambda$-geodesic flow. The lemma allows one to write down the solution of \eqref{att_trans_eq} as an integral representation using a matrix integrating factor.
\begin{lemma}
Let $(M, g, \lambda)$, $\lambda \in C^{\infty}(S M)$, be a compact Riemannian surface with strictly $\lambda$-convex boundary and nontrapping $\lambda$-geodesic flow. If the function $R: S M \rightarrow G L(n, \mathbb{C})$ satisfies $(X+\lambda V) R+\mathcal{A} R=0$, then
$$
u^f(x, v)=R(x, v) \int_0^{\tau(x, v)}\left(R^{-1} f\right)\left(\varphi_t(x, v)\right) d t \ \text { for }(x, v) \in S M .
$$
\end{lemma}
\begin{proof}
Using $(X+\lambda V) R^{-1}=R^{-1} \mathcal{A}$, see Lemma \ref{lm:R:equation}, and $(X+\lambda V) u^f+\mathcal{A} u^f=-f$, we obtain
\begin{align}
    (X+\lambda V)(R^{-1} u^f  )&= (X+\lambda V)(R^{-1} )u^f+R^{-1} (X+\lambda V)u^f\notag\\
    &= R^{-1} \mathcal{A} u^f-R^{-1}(\A u^f +f)=-R^{-1}f. \label{eq:p1:1}
\end{align}
Since $R^{-1}u^f|_{\pnsm}=0$, one immediately has
\begin{align}\label{eq:u:int:form:R}
    u^f(x,v)=R(x, v) \int_0^{\tau(x, v)}\left(R^{-1} f\right)\left(\varphi_t(x, v)\right) d t \ \text { for }(x, v) \in S M
\end{align}
and hence the lemma follows.
\end{proof}
We state a regularity result for the transport equation for $\lambda\in C^{\infty}(SM)$. In the case where $\lambda=0$, this is proved in \cite[Proposition 5.2]{Paternain:Salo:Uhlmann:2012}, and when $\lambda \in C^{\infty}(M)$, it is proved in \cite[Proposition 4.3]{Ainsworth:2013}. For the thermostatic case, see \cite[Proposition 4.3]{Assylbekov:Rea:Arxiv:2021}. The proof is similar in nature and, for the sake of completeness, is given in Appendix \ref{sec_app2}.
\begin{lemma}\label{lm:reg:u_A_Lam} 
Let $(M, g, \lambda)$, $\lambda \in C^{\infty}(S M)$, be a compact Riemannian surface with strictly $\lambda$-convex boundary and nontrapping $\lambda$-geodesic flow. If $I_\A^\lambda(f)=0$ for $f\in C^{\infty}(SM;\mathbb C^n)$, then $u^f:SM \to \mathbb C^n$ is a smooth function.  
\end{lemma}

\subsection{Pseudolinearization identity}
Let $\mathcal{A}, \mathcal{B} \in C^{\infty}(SM, \mathbb{C}^{n \times n})$ be two smooth matrix-valued functions on the unit tangent bundle $SM$. Our aim in this subsection is to establish a relation between $C^{\lambda}_{\mathcal{A}}$ and $C^{\lambda}_{\mathcal{B}}$ using a specific attenuated $\lambda$-ray transform. To this end, we define a map $E(\A,\B) : SM \to \operatorname{End}(\mathbb{C}^{n \times n})$ by setting 
\[E(\A,\B)U:=\A U-U\B,\]
where $E(\A,\B)$ denotes the linear endomorphisms of $\mathbb{C}^{n\times n}.$
Next proposition is a generalisation of \cite[Proposition 3.5]{Paternain:Salo:2020:a} to the case of the attenuated $\lambda$-ray transform.
\begin{lemma}\label{lm:psedo:ide}
Let $(M, g, \lambda)$, $\lambda \in C^{\infty}(S M)$, be a compact Riemannian surface with strictly $\lambda$-convex boundary and nontrapping $\lambda$-geodesic flow. Given any two matrix attenuation functions $\A,\B\in C^{\infty}(SM,\C^{n\times n})$, we have
\begin{equation}\label{eq:psedo:ide}    C_{\A}^{\lambda}\left[C^{\lambda}_{\B}\right]^{-1}=\operatorname{Id}+I_{E(\A,\B)}^{\lambda}(\A-\B)
\end{equation}
where $I_{E(\A,\B)}^{\lambda}$ denotes the attenuated $\lambda$-ray transform with matrix attenuation $E(\A,\B)$.
\end{lemma}
\begin{proof}
    Similar to the proof of \cite[Proposition 3.5]{Paternain:Salo:2020:a}, we consider
   \begin{align*}
       \left\{\begin{array}{l}
(X+\lambda V) U_{\A}+\A U_{\A}=0 \\
\left.U_{\A}\right|_{\partial_{-} S M}=\mathrm{Id}
\end{array}\right.
   \end{align*}
   and  \begin{align*}
       \left\{\begin{array}{l}
(X+\lambda V) U_{\B}+\B U_{\B}=0 \\
\left.U_{\B}\right|_{\partial_{-} S M}=\mathrm{Id}.
\end{array}\right.
   \end{align*}
 Using Lemma \ref{lm:R:equation}, we have 
   \begin{align}\label{eq:product-rule}
       (X+\lambda V)(U_\A U_{\B}^{-1})&= [(X+\lambda V)(U_\A)]U_{\B}^{-1}+U_\A[(X+\lambda V)(U_{\B}^{-1})]\\
       &=-\A U_\A U_{\B}^{-1}+U_\A U_{\B}^{-1}\B.
   \end{align}
Simplifying further, we obtain  
% \[
%   (X+\lambda V)[(U_\A U_{\B}^{-1})]+ \A U_\A U_{\B}^{-1} - U_\A U_{\B}^{-1}\B=0.
% \]
\[
  (X + \lambda V) \left[ U_\mathcal{A} U_{\mathcal{B}}^{-1} \right] + \mathcal{A} U_\mathcal{A} U_{\mathcal{B}}^{-1} - U_\mathcal{A} U_{\mathcal{B}}^{-1} \mathcal{B} = 0.
\]
This gives
\[
(X+\lambda V)\left[U_\A U_{\B}^{-1}-\operatorname{Id}\right]+ \A (U_\A U_{\B}^{-1}-\operatorname{Id}) - (U_\A U_{\B}^{-1}-\operatorname{Id})\B=-\A+\B
\]
and hence if we define $W := U_\A U_{\B}^{-1}-\operatorname{Id}$, then $W$ satisfies
\begin{align}\label{eq:W:attenuated:transport}
      \begin{split}
           \left\{\begin{array}{l}
(X+\lambda V)W + E(\A,\B)W=-(\A-\B) \\
\left.W\right|_{\partial_{-} S M}=0.
\end{array}\right.
      \end{split}
   \end{align}
Using Definition \ref{def:aLray}, we can write
\begin{equation}
I_{E(\A,\B)}^{\lambda}(\A - \B)=W|_{\psm} = U_\A U_{\B}^{-1} - \operatorname{Id}|_{\partial_{+}SM}.
\end{equation}
Since $C^{\lambda}_\A = U_\A|_{\partial_{+}SM}$ and $C^{\lambda}_\B = U_\B|_{\partial_{+}SM}$, see Definition \ref{def:scat}, the above identity becomes
\begin{equation}
I_{E(\A,\B)}^{\lambda}(\A - \B) = C^{\lambda}_\A \left[C^{\lambda}_\B\right]^{-1} - \operatorname{Id}
\end{equation}
and hence the lemma follows.
\end{proof}
The identity \eqref{eq:psedo:ide} is often called the pseudolinearization  identity.  The following proposition generalizes \cite[Proposition 3.7]{Paternain:Salo:2020:a} to the case of $\lambda$-geodesics.
\begin{proposition}\label{proposition:gaugeU}
Let $(M, g, \lambda)$, $\lambda \in C^{\infty}(S M)$, be a compact Riemannian surface with strictly $\lambda$-convex boundary and nontrapping $\lambda$-geodesic flow. Given any two matrix attenuation functions $\mathcal{A}, \mathcal{B} \in C^{\infty}\left(S M, \mathbb{C}^{n \times n}\right)$, we have $C^{\lambda}_{\mathcal{A}}=C^{\lambda}_{\mathcal{B}}$ if and only if there exists a smooth $U: S M \rightarrow G L(n, \mathbb{C})$ with $\left.U\right|_{\partial S M}=\mathrm{Id}$ such that
$$
\mathcal{B}=U^{-1} (X+\lambda V) U+U^{-1} \mathcal{A} U .
$$
\end{proposition}
\begin{proof}
Let $U$ be a smooth function satisfying the given conditions and define $Z:=U U_{\mathcal{B}}$.
Now, 
% \begin{align*}
%     (X+\lambda V) Z+\mathcal{A} Z&= \left((X+\lambda V) U\right)U_{\mathcal{B}}+U\left((X+\lambda V) U_{\mathcal{B}}\right)+\A UU_{\B}\\
%     &=(U\B-\A U)U_{\B}-U \B U_{\B}+\A U U_{\B}=0.
% \end{align*}
\begin{align*}
    (X + \lambda V) Z + \mathcal{A} Z &= \left( (X + \lambda V) U \right) U_{\mathcal{B}} + U \left( (X + \lambda V) U_{\mathcal{B}} \right) + \mathcal{A} U U_{\mathcal{B}} \\
    &= \left( U \mathcal{B} - \mathcal{A} U \right) U_{\mathcal{B}} - U \mathcal{B} U_{\mathcal{B}} + \mathcal{A} U U_{\mathcal{B}} \\
    &= 0.
\end{align*}
That is $Z$ satisfies both $(X+\lambda V) Z+$ $\mathcal{A} Z=0$ and $\left.Z\right|_{\partial_{-} S M}=\mathrm{Id}$. It follows that $Z=U_{\mathcal{A}}$ and consequently $C_{\mathcal{A}}=C_{\mathcal{B}}$.

Conversely, assume that the nonabelian $\lambda$-ray transforms agree for $\mathcal{A}$ and $\mathcal{B}$. In the proof of Proposition \ref{lm:psedo:ide}, the function $W=U_{\mathcal{A}} U_{\mathcal{B}}^{-1}-\mathrm{Id}$ has a zero boundary value. By Lemma \ref{lm:reg:u_A_Lam}, $W$ is smooth. Therefore, $U:=U_{\mathcal{A}}U^{-1}_{\mathcal{B}}=W+\operatorname{Id} $ is smooth and the function $U$ satisfies $U|_{\partial_{-}SM} = \mathrm{Id}$. Using \eqref{eq:product-rule}, we obtain
\begin{align*}
     U^{-1}(X+\lambda V)U + U^{-1}\mathcal{A} U&= U_{\B}U_{\A}^{-1}(X+\lambda V)(U_\A U_{\B}^{-1})+U_{\B}U_{\A}^{-1}\A U_\A U_{\B}^{-1}\\
     &=U_{\B}U_{\A}^{-1}(-\A U_\A U_{\B}^{-1}+U_\A U_{\B}^{-1}\B)+U_{\B}U_{\A}^{-1}\A U_\A U_{\B}^{-1}\\
     &=\B.
\end{align*}
That is $U$ satisfies the required equation.  
\end{proof}

\begin{remark} All the results in this section would generalize to smooth families of unit speed curves on the higher dimensional manifolds as well, when certain natural assumptions were made. The related flow should be smooth, nontrapping, and boundary should be strictly convex with respect to the flow. We do not pursue this matter further here, as the main results are specialized to surfaces, and we want to keep the presentation as straightforward as possible.
\end{remark}

\section{Proofs of Main results}\label{se_proof}
In \cite[Lemma 5.1]{Paternain:Salo:2020:a}, it is proved for any skew-Hermitian $\B\in C^{\infty}(SM,\fu(n))$ with $\B\in \bigoplus_{k \geq-1} \Omega_k$, it holds that $\B\in \Omega_{-1} \oplus \Omega_0 \oplus \Omega_1$. Similarly to their argument, one can also prove that if $\B\in C^{\infty}(SM,\fu(n))$ and $\B\in \bigoplus_{k \leq 1} \Omega_k$, then $\B\in \Omega_{-1} \oplus \Omega_0 \oplus \Omega_1$. 
\begin{lemma}\label{lm:skew_hermitian}
    Let $M$ be a Riemannian surface with or without boundary. Let $\B\in C^{\infty}(SM,\mathfrak{gl}(n))$. If $\B$ is skew-Hermitian, i.e., $\B\in C^{\infty}(SM,\mathfrak{u}(n))$, and $\B\in \oplus_{k\ge -m}\Omega_k$, then $\B\in \oplus_{-m\le k\le m}\Omega_k $.
\end{lemma}
\begin{proof}
  In the Fourier modes, one can write $\B=\sum_{k\ge -m}\B_k$. Note that
    \begin{align*}
        \B^*=\left(\sum_{k\ge -m}\B_k\right)^*=\sum_{k\ge -m}\B_k^*
    \end{align*}
   with $\B_k^*\in \Omega_{-k}$. This implies $\B^*\in \oplus_{k\le m} \Omega_k$. From skew-Hermitian property, $\B^*=-\B=-\sum_{k\ge -m}\B_k$. This implies $\B\in \oplus_{-m\le k\le m}\Omega_k $. 
\end{proof}
\begin{remark}
    Similarly, if $\B$ is skew-Hermitian and $\B \in \oplus_{k\le m}\Omega_k$, then $\B \in \oplus_{-m\le k\le m}\Omega_k$.
\end{remark}

In \cite[Theorem 4.2]{Paternain:Salo:2020:a}, it is proved that, given any smooth map $R: S M \rightarrow G L(n, \mathbb{C})$, one can always find smooth maps $U:S M \rightarrow U(n)$ and $F: S M \rightarrow G L(n, \mathbb{C})$ where the map $F$ is fiberwise holomorphic with a fiberwise holomorphic inverse, and $R=FU$. Additionally, one can decompose $R=\widetilde F\widetilde U$, where $\widetilde U:SM \to U(n)$ is smooth and $\widetilde F: S M \rightarrow G L(n, \mathbb{C})$ is smooth, fiberwise antiholomorphic, and has a fiberwise antiholomorphic inverse. The following lemma generalizes \cite[Lemma 5.2]{Paternain:Salo:2020:a}. For the sake of completeness, we include the proof here.
\begin{lemma}\label{lm:loop:skew}
Let $(M, g, \lambda)$, $\lambda \in C^{\infty}(S M)$, be a compact Riemannian surface with strictly $\lambda$-convex boundary and nontrapping $\lambda$-geodesic flow. Let $\A\in C^{\infty}(SM,\mathfrak{gl}(\C))$. Let $R:SM \to GL(n,\C)$ be a smooth function solving $(X+\lambda V+\A)R=0$ with a loop group factorisation $R=FU$, where $F$ is fiberwise holomorphic with a fiberwise holomorphic inverse and $U$ is unitary. Then 
\begin{equation}
        \B=F^{-1}(X+\lambda V+\A) F
        \end{equation}
    is skew-Hermitian. In addition, if $\A \in \bigoplus_{k \geq-m} \Omega_k$ and $\lambda \in \bigoplus_{k \geq-m-1} \Omega_k$ for $m\ge 1$, then $\B\in \bigoplus_{-m\leq k \leq m} \Omega_k$.
\end{lemma}
\begin{proof}
Using our assumptions on $R$, we obtain
\begin{align}\label{eq:a:2}
    0=(X+\lambda V+\A) R= ((X+\lambda V)F)U+F(X+\lambda V)U+\A FU.
\end{align}
By multiplying $F^{-1}$ we have
% \begin{align*}
%    0= F^{-1}((X+\lambda V)F)U+(X+\lambda V))U+F^{-1}\A FU.
% \end{align*}
\begin{align*}
    0 = F^{-1} \left( (X + \lambda V) F \right) U + (X + \lambda V) U + F^{-1} \mathcal{A} F U.
\end{align*}
This implies
\begin{align}\label{eq:a:1}
    (X+\lambda V+\B)U=0
\end{align}
where $\B=F^{-1}(X+\lambda V+\A) F$. By rearranging \eqref{eq:a:1}, we get 
\begin{equation}
    \B= -((X+\lambda V)U )U^{-1}.
\end{equation}
As $U$ is unitary, we have
\begin{align*}
    \left(((X+\lambda V)U )U^{-1}\right)^*&=(U^*)^*(X+\lambda V)U^{*}\\
    &=U(X+\lambda V)U^{-1}\\
    &=-((X+\lambda V)U )U^{-1}.
\end{align*}
This proves that $\B$ is skew-Hermitian. 

Now let us assume that $\A$ is in $\bigoplus_{k \geq-m} \Omega_k$ and $\lambda \in \bigoplus_{k \geq-m-1} \Omega_k$ for $m\ge 1$. Under this assumption, we have the following mapping property: \[X+\lambda V:\oplus_{k \geq 0} \Omega_k \rightarrow \oplus_{k \geq-m} \Omega_k. \]
Since $F$ and $F^{-1}$ are holomorphic, i.e., $F,F^{-1}\in  \oplus_{k \geq 0} \Omega_k$, we have $F^{-1}(X+\lambda V)F\in \oplus_{k \geq-m} \Omega_k$. Similarly, since  $\A\in \bigoplus_{k \geq-m} \Omega_k$, we have $F^{-1}\A F \in \bigoplus_{k \geq-m} \Omega_k$.  This implies $\B\in  \bigoplus_{k \geq-m} \Omega_k$. By Lemma \ref{lm:skew_hermitian}, we have $\mathcal{B} \in \bigoplus_{-m \leq k \leq m} \Omega_k$.
\end{proof}

In similar spirit, one can prove the following lemma using antiholmorphic and unitary loop group decomposition.
\begin{lemma}
   Let $(M, g, \lambda)$, $\lambda \in C^{\infty}(S M)$, be a compact Riemannian surface with strictly $\lambda$-convex boundary and nontrapping $\lambda$-geodesic flow. Let $\A\in C^{\infty}(SM,\mathfrak{gl}(\C))$. Let $R:SM \to GL(n,\C)$ be a smooth function solving $(X+\lambda V+\A)R=0$ with a loop group factorisation $R=\widetilde F\widetilde U$, where $\widetilde F$ is fiberwise antiholomorphic with a fiberwise antiholomorphic inverse and $\widetilde U$ is unitary. Then 
    \begin{equation}
        \B=\widetilde F^{-1}(X+\lambda V+\A) \widetilde F
    \end{equation}
    is skew-Hermitian. In addition, if $\A \in \bigoplus_{k \leq m} \Omega_k$ and $\lambda \in \bigoplus_{k \leq m+1} \Omega_k$ for $m\ge 1$, then  $\B\in \bigoplus_{-m\le k \leq m} \Omega_{k}$.
\end{lemma}
The following theorem generalizes \cite[Theorem 5.5]{Paternain:Salo:2020:a} to the case of $\lambda$-geodesic flow for $\lambda \in \bigoplus_{-2 \leq k \leq 2} \Omega_k$. Recall that in the case of $\lambda \in \bigoplus_{-2 \leq k \leq 2} \Omega_k$, the vector field $ X + \lambda V$ has the following mapping property:
\[
X + \lambda V : \oplus_{k\geq 0}\Omega_{k}\to \oplus_{k\geq -1}\Omega_{k}.
\]
\begin{theorem}\label{holo} 
Let $(M,g,\lambda)$, $\lambda \in C^\infty(SM)$, be a compact Riemanian surface with strictly $\lambda$-convex boundary and nontrapping $\lambda$-geodesic flow. Let $\lambda \in \bigoplus_{-2 \leq k \leq 2} \Omega_k$ and $\A\in C^{\infty}(SM,\mathfrak{gl}(n,\C))$ with $\A \in \Omega_{-1} \oplus \Omega_0 \oplus \Omega_1$. Additionally, assume that Conjecture \ref{conjecture 3} is true.
Let $u\in C^{\infty}(SM,\mathbb{C}^{n})$ satisfy $u|_{\partial SM}=0$ and
\[(X+\lambda V)u+\A u=-f\in \oplus_{k\geq -1}\Omega_{k}.\]
Then $u$ is holomorphic.
\label{thm:holomorphic}
\end{theorem}
\begin{proof} 
From Lemma \ref{lm:R:equation}, there exists a smooth function $R$ such that $(X+\lambda V)R + \mathcal{A}R = 0$ and $(X+\lambda V)R^{-1} - R^{-1}\mathcal{A} = 0$. As in Lemma \ref{lm:loop:skew}, $R$ can be written as $R = FU$, where $F$ is fiberwise holomorphic with a fiberwise holomorphic inverse, and $U$ is unitary.
Now using $\B$ as in Lemma \ref{lm:loop:skew}, we have
\begin{align*}
    (X+\lambda V)&(F^{-1}u)+\B(F^{-1}u)\\
    &= -F^{-1}((X+\lambda V)F)F^{-1}u+F^{-1}(X+\lambda V)u+ F^{-1}((X+\lambda V+\mathcal{A}) F) F^{-1}u\\
    &=F^{-1}(X+\lambda V+\A)u=-F^{-1}f.
\end{align*}
and $F^{-1}u|_{\partial SM}=0$. Since $f \in \oplus_{k \geq-1} \Omega_k$ and $F^{-1}$ is holomorphic, it follows that $F^{-1}f \in \oplus_{k \geq-1} \Omega_k$.
Similar to the proof of \cite[Theorem 5.5]{Paternain:Salo:2020:a}, we consider $q:=\sum_{-\infty}^{-1}(F^{-1}u)_k$. From 
\begin{align*}
    (X+\lambda V)(F^{-1}u)+\B(F^{-1}u)=-F^{-1}f\in \oplus_{k \geq-1} \Omega_k,
\end{align*}
we obtain
\begin{align*}
    (X+\lambda V)q+\B q\in\Omega_{-1}\oplus\Omega_{0}.
\end{align*}
Since $q|_{\partial SM}=0$, $\B$ is skew-Hermitian and $\B \in \Omega_{-1} \oplus \Omega_0 \oplus \Omega_1$, it follows from Conjecture \ref{conjecture 3} that $$ (X+\lambda V + \B) p = (X+\B) p = (X+\lambda V +\B)q$$ for some $p \in C^\infty(M)$ with $p|_{\partial M}=0$. Therefore $r := p-q$ solves the equation
$(X+\lambda V + \B) r=0$
with the boundary condition $r|_{\partial SM} = 0$. Since the solution to this problem is unique due to the uniqueness of solutions for a linear system of first order ODEs (using the fact that $\lambda$-geodesic flow is nontrapping), we find that $r=0$. Therefore $q=p \in \Omega_0$, which gives that $q=0$.

This implies that $F^{-1}u$ is holomorphic and hence $u=F(F^{-1}u)$ is also holomorphic.
\end{proof}

\begin{remark}\label{remark:obstruction} If $\lambda$ has Fourier modes of three or higher, then $\B$ given by Lemma \ref{lm:loop:skew} possibly has Fourier modes of two or higher, and one may not conclude the desired uniqueness result with this method. One would instead reduce to a transport inverse problem with an attenuation having higher Fourier modes.
\end{remark}

\begin{proof}[Proof of Theorem \ref{conj2_conj1}]
From Proposition \ref{proposition:gaugeU}, we infer that $C_{A,\Phi}^{\lambda} = C_{B,\Psi}^{\lambda}$ implies the existence of a smooth function $U: SM \to GL(n, \mathbb{C})$ such that $U|_{\partial SM} = \id$ and
\begin{equation} \label{eq:geq}
\B = U^{-1}(X + \lambda V)U + U^{-1}\A U,
\end{equation}
where $\B(x,v) = B_{x}(v) + \Psi(x)$ and $\A(x,v) = A_{x}(v) + \Phi(x)$. We can rephrase this information in terms of an attenuated $\lambda$-ray transform. If we define $W = U - \id$, then $W|_{\partial SM} = 0$ and \[(X+\lambda V)W+\A W-W\B=-(\A-\B).\]
That is, $W$ satisfies \eqref{eq:W:attenuated:transport}. Consequently, $W$ is associated with the attenuated $\lambda$-ray transform $I_{E(\A,\B)}^{\lambda}(\A - \B)$, and if $C_{\A}^{\lambda} = C_{\B}^{\lambda}$, then this transform vanishes by Lemma \ref{lm:psedo:ide}. Note that $\A - \B \in \Omega_{-1} \oplus \Omega_{0} \oplus \Omega_{1}$.

We introduce a new connection $\hat A$ on the trivial bundle $M \times \mathbb{C}^{n \times n}$ and a new Higgs field $\hat \Phi$ as follows: for any matrix $H \in \mathbb{C}^{n \times n}$, we define $\hat A(H) := AH - HB$ and $\hat \Phi(H) := \Phi H - H\Psi$. Observe that $E(\A,\B) = \hat A + \hat \Phi$, which is an arbitrary attenuation pair as assumed in Conjecture \ref{conjecture 2}. Applying Conjecture \ref{conjecture 2}, we deduce the existence of a function $\widetilde W \in C^{\infty}(M)$ satisfying $\widetilde W|_{\partial M} = 0$ and 
    \begin{align*}
        (X + \lambda V + \hat A + \hat \Phi)\widetilde W &= (X + \hat A + \hat \Phi)\widetilde W = (X + \lambda V + \hat A + \hat \Phi) W.
    \end{align*}
    Consequently, the function $r := \widetilde W - W$ solves the equation $(X + \lambda V + \hat A + \hat \Phi)r = 0$ with $r|_{\partial SM} = 0$. Following a similar argument as used in the proof of Theorem \ref{thm:holomorphic}, we conclude that $r = 0$. Therefore, $\widetilde W=W\in \Omega_0$, indicating that $W$ smooth function on $M$. Thus, $U=W+\id$, which depends only on $x$, and setting $u(x) = U_{0}(x) = U(x)$, equation \eqref{eq:geq} can be reformulated as $B = u^{-1}du + u^{-1}Au$ and $\Psi = u^{-1}\Phi u$, by examining solely the components of degree $0$ and $\pm 1$. This proves the existence of $u \in \mathcal{G}$ such that $(A, \Phi)\cdot{u}=(B, \Psi)$.
\end{proof}

\begin{proof}[Proof of Theorem \ref{them_conjecture 2}]
Consider an arbitrary attenuation pair $(A, \Phi)$, where
\begin{equation}
A \in \Omega^1(M, \mathfrak{g l}(n, \mathbb{C})), \qquad \Phi \in C^{\infty}(M, \mathfrak{g l}(n, \mathbb{C})).
\end{equation}
Define $\A(x, v) = A(x, v) + \Phi(x)$. If $I_{\A}^{\lambda}f = 0$, then by Lemma \ref{lm:reg:u_A_Lam}, we have $u \in C^{\infty}(SM)$, $u|_{\partial SM} = 0$, and
\begin{equation}\label{eq:q1}
(X + \lambda V)u + \A u = -f \in \Omega_{-1} \oplus \Omega_0 \oplus \Omega_1.
\end{equation}
Given that $\A \in \oplus_{-1 \leq k \leq 1} \Omega_k$ and $\lambda \in \bigoplus_{-2 \leq k \leq 2} \Omega_k$, Theorem \ref{holo} implies $u$ is holomorphic. The complex conjugate takes $w_k\in \Omega_k$ to $\overline{w_k}\in \Omega_{-k}$. Taking the complex conjugate of equation \eqref{eq:q1}, and considering that $\overline{\A}, \overline{f} \in \oplus_{-1 \leq k \leq 1} \Omega_k$ and $\overline{\lambda} \in \bigoplus_{-2 \leq k \leq 2} \Omega_k$, we deduce that $\overline{u}$ is also holomorphic (using Theorem \ref{holo}). This leads to the conclusion that $u = u_0$ is solely dependent on $x$. Setting $p = -u$, we then obtain $f = (X + \A)p$.
\end{proof}
From \cite[Theorem 1.2]{Ainsworth:2013} and Theorem \ref{them_conjecture 2}, we can easily deduce Theorem \ref{them_conjecture 3}.

\appendix
\section{Some basic properties of the nontrapping \texorpdfstring{$\lambda$}{lambda}-geodesic flows}
\label{sec:appendix-nontrappingness}

Our main goal in this appendix is to show that the existence of nontrapping $\lambda$-geodesic flow is an open condition. For that, we construct a smooth function $f$ on $SM$ satisfying the property that $(X+\lambda V)f > 0$. For the case of $\lambda=0$, this is shown in \cite[Proposition 3.3.1]{GIP2D}. The proof for the general case follows similarly. For the sake of completeness, we present it in this appendix. Before that, we restate one technical lemma from \cite{GIP2D}.

\begin{lemma}[{\cite[Lemma 3.2.10]{GIP2D}}]\label{lm:3.2.10.gip}
   Let $h(t, y)$ be a smooth function near $\left(0, y_0\right)$ in $\mathbb{R} \times \mathbb{R}^N$. If
$$
h\left(0, y_0\right)=0, \quad \partial_t h\left(0, y_0\right)=0, \quad \partial_t^2 h\left(0, y_0\right)<0,
$$
then one has
$$
h(t, y)=0 \text { near }\left(0, y_0\right) \text { when } h(0, y) \geq 0 \quad \Longleftrightarrow \quad t=Q\left( \pm \sqrt{a(y)}, y\right)
$$
where $Q$ is a smooth function near $\left(0, y_0\right)$ in $\mathbb{R} \times \mathbb{R}^N$, a is a smooth function near $y_0$ in $\mathbb{R}^N$, and $a(y) \geq 0$ when $h(0, y) \geq 0$. Moreover, $Q\left(\sqrt{a(y)}, y\right) \geq Q\left(-\sqrt{a(y)}, y\right)$ when $h(0, y) \geq 0$.
\end{lemma}
The following lemma generalizes \cite[Lemma 3.2.9]{GIP2D} in the context of the forward and backward travel times of $\lambda$-geodesic flow.
\begin{lemma}\label{lm:taupm}
    Let $(M,g,\lambda)$, $\lambda\in C^{\infty}(SM)$, be a compact Riemmanian surface. Let $(x_0,v_0)\in \partial_0SM$ and suppose that $\partial M$ is strictly $\lambda$-convex near $x_0$. Assume that $M$ is isometrically embedded in a closed manifold $N$ and $\lambda_0$ is smooth extension of $\lambda$ to $N$. Then, near $(x_0,v_0)$ in $SM$, one has\begin{align}\label{eq:taupm}
        \begin{split}
            \tau^+(x, v) &= Q\left(\sqrt{a(x, v)}, x, v\right), \\
        -\tau^-(x,v) &= Q\left(-\sqrt{a(x, v)}, x, v\right),
        \end{split}
    \end{align} 
    where $Q$ is a smooth function near $(0, x_0,v_0)$ in $ \mathbb R\times SN$, and $a$ is a smooth function near $(x_0,v_0)$ in $SN$ with $a\ge 0$ in $SM$.
\end{lemma}
\begin{proof}
    Following the approach of \cite[Lemma 3.2.9]{GIP2D}, we define $h:\mathbb R\times SN\to \mathbb R$, $h(t, x, v):=\rho\left(\gamma_{x, v}(t)\right)$, where $\rho$ is a boundary defining function satisfying, $\rho|_{\operatorname{int} M}> 0$, $\rho|_{\partial M}=0$ and $\rho|_{N\setminus M}<0$. According to \cite[Lemma A.12]{Jathar:Kar:Railo:2023:arxiv} and \cite[Lemma 3.15]{Jathar:Kar:Railo:2023:arxiv}, for any $(x_0,v_0)\in \partial_0SM$, we have the following:
    \begin{align*}
       h\left(\pm \tau^{ \pm}\left(x_0, v_0\right), x_0, v_0\right)&=  h\left(0, x_0, v_0\right)=0,\\ \partial_t h\left(\pm \tau^{ \pm}\left(x_0, v_0\right), x_0, v_0\right)&=\partial_t h\left(0, x_0, v_0\right)=0, \\ \partial_t^2 h\left(\pm \tau^{ \pm}\left(x_0, v_0\right), x_0, v_0\right)&=\partial_t^2 h\left(0, x_0, v_0\right)<0.
    \end{align*}
           Also, $h(\pm \tau^{\pm}(x,v), x,v)=0$ for any $(x,v)\in SM$ near $(x_0,v_0)$. Note that $\tau^+(x,v)\ge 0\ge -\tau^{-}(x,v)$. Applying Lemma \ref{lm:3.2.10.gip}, we express $\tau^{+}(x, v)$ and $-\tau^{-}(x, v)$ as in equation \eqref{eq:taupm}, confirming that $Q$ and $a$ have the stated properties.  
\end{proof}

\begin{lemma}\label{lm:smooth:ext:tau}
 Let $(M, g, \lambda), \lambda \in C^{\infty}(S M)$, be a compact Riemanian surface with strictly $\lambda$-convex boundary and nontrapping $\lambda$-geodesic flow. Define $\tilde{\tau}: SM\to \mathbb R$ by 
\begin{equation}
    \tilde{\tau}(x, v):=\tau^+(x, v)-\tau^-(x,v)
\end{equation}
Then $\tilde{\tau} \in C^{\infty}( S M)$.
\end{lemma}
\begin{proof}
 Assume that $(M, g)$ is isometrically embedded in a closed manifold $(N, g)$ of the same dimension and $\lambda_0$ is a smooth extension of $\lambda$ in $N$.
From \cite[Lemma 3.15]{Jathar:Kar:Railo:2023:arxiv}, it suffices to show smoothness of $\widetilde\tau$ in the neighbourhood of $\partial_0 S M$.

From Lemma \ref{lm:taupm}, we have 
\begin{align*}
    \widetilde\tau(x,v)=Q\left(\sqrt{a(x, v)}, x, v\right)+Q\left(-\sqrt{a(x, v)}, x, v\right)
\end{align*}
     where $Q$ is a smooth function near $(0, x_0, v_0)$ in $\mathbb{R} \times SN$ and $a$ is a smooth function near $(x_0, v_0)$ in $SN$, and $a \geq 0$ in $SM$ 
     Now, as $Q(r,x,v)+Q(-r,x,v)$ is an even function in $r$, using \cite[Theorem 1]{Whitney:1943}, there exists a smooth function $H$ near $(0,x_0,v_0)$ such that $Q(r,x,v)+Q(-r,x,v)=H(r^2,x,v)$. This implies $\widetilde \tau(x,v)=H(a(x,v),x,v)$, which proves the smoothness of $\widetilde \tau$.
\end{proof}
Finally, we prove that having nontrapping $\lambda$-geodesic flow is an open condition. 
\begin{lemma}\label{lm:open:nontrapped}
    Let $(M, g, \lambda), \lambda \in C^{\infty}(S M)$, be a compact Riemanian surface with strictly $\lambda$-convex boundary. The $\lambda$-geodesic flow is nontrapping if and only if there exists a function $f \in C^{\infty}(S M)$ such that the inequality $(X+\lambda V) f>0$ holds.
\end{lemma}\begin{proof}
    First, assume that the $\lambda$-geodesic flow is nontrapping on $M$. By Lemma \ref{lm:smooth:ext:tau}, the function $\widetilde{\tau}$ is smooth on $S M$. Consider the function $f=-\widetilde{\tau}$. Then, we have
    % \begin{align*}
    %     (X+\lambda V)f(x,v) &= \left.-\frac{d}{dt}\left(\tau^+(\phi_t(x,v)) - \tau^-(\phi_{t}(x,v)) \right)\right|_{t=0} \\
    %     &= -\left.\frac{d}{dt}\left(\tau^+(x,v) - t - (\tau^-(x,v) + t) \right)\right|_{t=0} = 2 > 0.
    % \end{align*}
     \begin{align*}
 (X + \lambda V) f(x, v) &= -\left. \frac{d}{dt} \left( \tau^+(\phi_t(x, v)) - \tau^-(\phi_t(x, v)) \right) \right|_{t=0} \\
    &= -\left. \frac{d}{dt} \left( \left[ \tau^+(x, v) - t \right] - \left[ \tau^-(x, v) + t \right] \right) \right|_{t=0}= 2 > 0.
    \end{align*}
   Conversely, assume there exists $f \in C^{\infty}(S M)$ such that $(X+\lambda V) f \geq c>0$ for some constant $c$. Suppose, for the sake of contradiction, that there exists a trapped $\lambda$-geodesic $\gamma$. Applying the fundamental theorem of calculus, we obtain $f\left(\phi_t(x, v)\right)-f(x, v) \geq c t$ for all $t>0$. This implies that $f\circ (\gamma,\dot{\gamma})$ is unbounded, which contradicts the compactness of $S M$. Hence, the assumption of the existence of a trapped $\lambda$-geodesic is wrong, completing the proof.
\end{proof}

    \section{Regularity of the primitive for the kernel of the attenuated \texorpdfstring{$\lambda$}{lambda}-ray transform}\label{sec_app2}
Given any function $w\in C^{\infty}(\partial_+SM,\mathbb C^n)$, consider the unique solution $w^{\sharp}:SM\to \mathbb C^n$ to the transport equation:
\begin{align*}
    \begin{cases}
        (X + \lambda V)(w^{\sharp}) = 0, \\
        w^{\sharp}|_{\partial_{+}(SM)} = w.
    \end{cases}
\end{align*}
Similar to the concept of scattering relation along geodesics, as defined in \cite[Definition 3.3.4]{GIP2D}, we define the $\lambda$-scattering relation $\alpha$ as a map which associates the initial point and direction of a $\lambda$-geodesic flow with its endpoint and direction.
More precisely, let $(M, g, \lambda)$, with $\lambda \in C^{\infty}(SM)$, be a compact Riemannian surface with strictly $\lambda$-convex boundary and nontrapping $\lambda$-geodesic flow. Then, the $\lambda$-scattering relation $\alpha: \partial SM \to \partial SM$ is given by $\alpha(x, v) := \phi_{\tilde\tau}(x, v)$, where $\tilde\tau$ is defined as in Lemma \ref{lm:smooth:ext:tau}.

Similar to the approach in \cite{Pestov:Uhlmann:2005}, we introduce the operator $E:C^{\infty}(\partial_+SM)\to C^{\infty}(\partial SM)$ defined as 
\begin{align*}
   Ew(x, v) = \begin{cases}
                  w(x, v) & \text{if } (x, v) \in \partial_{+}SM, \\
                  (w \circ \alpha)(x, v) & \text{if } (x, v) \in \partial_{-}SM.
              \end{cases}
\end{align*}
Additionally, we define the space
\begin{align*}
    C_{\alpha, \lambda}^{\infty}(\partial_+SM) := \{ w \in C^{\infty}(\partial_+SM) : Ew \in C^{\infty}(\partial SM) \}.
\end{align*}
The following lemma generalizes \cite[Lemma 1.1]{Pestov:Uhlmann:2005} or \cite[Theorem 5.1.1]{GIP2D} to the context of $\lambda$-geodesic flow.
\begin{lemma}\label{lm:space}
    Let $(M,g,\lambda)$, $\lambda \in C^{\infty}(SM)$, be a compact Riemannian surface with strictly $\lambda$-convex boundary and nontrapping $\lambda$-geodesic flow. Then 
    \begin{align*}
    C_{\alpha,\lambda}^{\infty}(\partial_+SM):=\{w\in C^{\infty}(\partial_+SM):w^{\sharp}\in C^{\infty}( SM)\}.
\end{align*}
\end{lemma}
\begin{proof}
    If $w^{\sharp} \in C^{\infty}(S M)$, then obviously $E w=\left.w^{\sharp}\right|_{\partial_{+} S M} \in C^{\infty}(\partial S M)$. Conversely, assume that $E w \in C^{\infty}(\partial S M)$. It suffices to show that $w^{\sharp} \in C^{\infty}(S M)$.

    Assume that $(M, g)$ is isometrically embedded in a closed manifold $(N, g)$ of the same dimension, as stated in \cite[Lemma 3.1.8]{GIP2D}. Extend $\lambda$ smoothly to the whole $SN$. Consider a smooth extension $\tilde{w}$ of $Ew$ into $SN$. Define $\mathfrak{F}(x, v, t) = \frac{1}{2}\tilde{w}(\phi_t(x, v))$, then
    \begin{align*}
        w^{\sharp}(x, v) &= \frac{1}{2}\left[\tilde{w}\left(\phi_{\tau^+(x, v)}(x, v)\right) + \tilde{w}\left(\phi_{-\tau^-(x, v)}(x, v)\right)\right] \\
        &= \mathfrak{F}\left(x, v, \tau^+(x, v)\right) + \mathfrak{F}\left(x, v, -\tau^-(x, v)\right).
    \end{align*}
    As $\tau^{\pm}$ is smooth on $SM \setminus \partial_0SM$, $w^{\sharp}$ is smooth on $SM \setminus \partial_0SM$. Fixing any point $(x_0, v_0) \in \partial_0SM$, and from Lemma \ref{lm:taupm}, we have
    \begin{align*}
        w^{\sharp}(x, v) &= \mathfrak{F}\left(x, v, Q(\sqrt{a(x, v)}, x, v)\right) + \mathfrak{F}\left(x, v, Q(-\sqrt{a(x, v)}, x, v)\right)
    \end{align*}
    near $(x_0, v_0)$ in $SM$, where $Q$ is a smooth function near $(0, x_0, v_0) \in \mathbb{R} \times SN$. Set $\mathcal G(r,x,v):=\mathfrak F(x,v,Q(r,x,v))$. This implies 
    \begin{align*}
      w^{\sharp}(x, v) &=\mathcal G\left(\sqrt{a(x,v)},x,v \right)+  \mathcal G\left(-\sqrt{a(x,v)},x,v \right)
    \end{align*}
    near $(x_0, v_0)$ in $SM$, where $\mathcal G$ is a smooth function near $(0, x_0, v_0) \in \mathbb{R} \times SN$. Now, as $\mathcal G(r,x,v)+\mathcal G(-r,x,v)$ is even function in $r$, by using \cite[Theorem 1]{Whitney:1943}, there exists smooth function $H$ near $(0,x_0,v_0)$ such that $\mathcal G(r,x,v)+\mathcal G(-r,x,v)=H(r^2,x,v)$. This indeed shows that 
    \begin{align*}
       w^{\sharp}(x, v) =H(a(x,v),x,v)  
    \end{align*}
    near $(x_0,v_0)$ in $SM$, which finally proves that $ w^{\sharp}(x, v)$ is smooth near the point $(x_0,v_0)$ in $SM$. As  $(x_0,v_0)$ is any arbitrary point in $\partial_0SM$, we have $ w^{\sharp}\in C^{\infty}(SM)$.
\end{proof}
\begin{proof}[Proof of Lemma \ref{lm:reg:u_A_Lam}]
    From \eqref{eq:u:int:form:R} and $I_{\A}^{\lambda}(f)=0$, we have
    \begin{align*}
        (X+\lambda V)(R^{-1}u^f)=-R^{-1}f \quad \text{ in } SM\quad  R^{-1}u^f|_{\partial SM}=0.
    \end{align*}
Similar to the proof of \cite[Theorem 5.3.6]{GIP2D}, by enlarging $M$ to $M_0$ (as defined in Definition \ref{def:smooth:lambda:extension}) with the travel time $\tau_0$, and also extending $R^{-1}f$ smoothly to $M_0$, consider
    \begin{align*}
        h(x, v) = \int_0^{\tau_0(x, v)} R^{-1}f(\phi_t(x, v)) \, dt
    \end{align*}
    for $(x, v) \in SM$. Note that $\tau_0|_{SM}$ is smooth, which implies $h\in C^{\infty}(SM)$. Also, $(h-R^{-1}u^f)|_{\partial SM}=h|_{\partial SM}\in C^{\infty}(\partial SM,\mathbb C^n) $. As $(X+\lambda V)(h-R^{-1}u^f)=0$ and $(h-R^{-1}u^f)|_{\partial SM}\in C^{\infty}(\partial SM,\mathbb C^n)$, then by Lemma \ref{lm:space}, $h - R^{-1}u^f$ is smooth in $SM$. The smoothness of $h$ in $SM$ implies that $R^{-1}u^f$ is smooth, which in turn proves that $u^f$ is smooth.
\end{proof}    

\subsection*{Acknowledgments} J.R. thanks Gabriel P. Paternain for many helpful discussions related to the topics of this work. S.R.J. and J.R. would like to thank the Isaac Newton Institute for Mathematical Sciences, Cambridge, UK, for support and hospitality during \emph{Rich and Nonlinear Tomography - a multidisciplinary approach} in 2023 where part of this work was done (supported by EPSRC Grant Number EP/R014604/1). The work of S.R.J. and J.R. was supported by the Research Council of Finland through the Flagship of Advanced Mathematics for Sensing, Imaging and Modelling (decision number 359183). S.R.J. acknowledges the Prime Minister's Research Fellowship (PMRF) from the Government of India for his PhD work.  M.K. was supported by MATRICS grant (MTR/2019/001349) of SERB. J.R. was supported by the Vilho, Yrjö and Kalle Väisälä Foundation of the Finnish Academy of Science and Letters.
\subsection*{Data availability statement} Data sharing not applicable to this article as no datasets were generated or analyzed during the current study.
\subsection*{Conflict of interest}
The authors declared no potential conflicts of interest with respect to the research, authorship, and/or publication of this article.

   \bibliography{math} 
	
	\bibliographystyle{alpha}

\end{document}